\documentclass[a4paper,12pt]{article}
\usepackage[top=2.5cm,bottom=2.5cm,left=2.5cm,right=2.5cm]{geometry}
\usepackage{cite, amsmath, amssymb}
\usepackage{amssymb}
\usepackage{graphicx}
\newenvironment{proof}{{\noindent \it Proof.}}{\hfill $\blacksquare$\par}
\usepackage{latexsym}
\usepackage{graphicx,booktabs,multirow}
\usepackage{tikz}
\usetikzlibrary{decorations.pathreplacing}
\usetikzlibrary{intersections}
\usepackage{enumerate}
\usepackage{graphicx,booktabs,multirow}
\usepackage{appendix}
\usepackage{diagbox}
\newtheorem{theorem}{Theorem}[section]

\newtheorem{lemma}[theorem]{Lemma}

\begin{document}
%\linenumbers
%\newgeometry{top=6cm,bottom=2.5cm,left=2.5cm,right=2.5cm}

\title{The spread of generalized reciprocal distance matrix}
\author{Hechao Liu$^1$, Yufei Huang$^{2,}$\thanks{Corresponding author}
 \\
{\small $^1$ School of Mathematical Sciences, South China Normal University,}\\ {\small Guangzhou, 510631, P. R. China}\\
 \small {\tt hechaoliu@m.scnu.edu.cn}\\
{\small $^2$ Department of Mathematics Teaching, Guangzhou Civil Aviation College,}\\ {\small Guangzhou, 510403, P. R. China}\\
 \small {\tt fayger@qq.com}\\
}
\date{}
\maketitle
\begin{abstract}
The generalized reciprocal distance matrix $RD_{\alpha}(G)$ was defined as $RD_{\alpha}(G)=\alpha RT(G)+(1-\alpha)RD(G),\quad 0\leq \alpha \leq 1.$
Let $\lambda_{1}(RD_{\alpha}(G))\geq \lambda_{2}(RD_{\alpha}(G))\geq \cdots \geq \lambda_{n}(RD_{\alpha}(G))$ be the eigenvalues of $RD_{\alpha}$ matrix of graphs $G$. Then the $RD_{\alpha}$-spread of graph $G$ can be defined as $S_{RD_{\alpha}}(G)=\lambda_{1}(RD_{\alpha}(G))-\lambda_{n}(RD_{\alpha}(G))$.
In this paper, we first obtain some sharp lower and upper bounds for the $RD_{\alpha}$-spread of graphs. Then we determine the lower bounds for the $RD_{\alpha}$-spread of bipartite graphs and graphs with given clique number.  At last, we give the $RD_{\alpha}$-spread of double star graphs.
Our results generalize the related results of the reciprocal distance matrix and reciprocal distance signless Laplacian matrix.
\end{abstract}

\hskip0.2cm

\maketitle

%%%%%% THIS PART MUST BE PLACED IMMEDIATELY AFTER THE \maketitle COMMAND
%%%%%% BACK TO ORIGINAL FOOTNOTES
\makeatletter
\renewcommand\@makefnmark%
{\mbox{\textsuperscript{\normalfont\@thefnmark)}}}
\makeatother
%%%%%%

%\newgeometry{top=2.5cm,bottom=2.5cm,left=2.5cm,right=2.5cm}
 \baselineskip=0.30in

\baselineskip=0.30in

%\restoregeometry
\section{Introduction}
\hskip 0.6cm
Let $G$ be a graph with vertex set $V(G)=\{v_{1},v_{2},\cdots,v_{n}\}$ and edge set $E(G)$.
The complement graph $\overline{G}$ is the graph with $V(\overline{G})=V(G)$ and $uv\in E(\overline{G})$ if and only if $uv\notin E(G)$.
Let $d_{G}(v_{i},v_{j})$ (or $d_{ij}$) be the distance between vertex $v_{i}$ and $v_{j}$ in $G$.
Denote by $RTr_{G}(v_{i})=\sum\limits_{v_{j}\in V(G)\setminus \{v_{i}\}}\frac{1}{d_{G}(v_{i},v_{j})}$
the vertex reciprocal transmission of vertex $v_{i}$ in $G$.
We call $G$ is a $k$-reciprocal transmission regular graph if $RTr_{G}(v_{i})=k$ for any $v_{i}\in V(G)$.
We usually omit the letter $G$ when it does not cause confusion. For example, we write $RTr_{G}(v_{i})$
as $RTr(v_{i})$ or $RTr_{i}$, write $\frac{1}{d_{G}(v_{i},v_{j})}$ as $\frac{1}{d(v_{i},v_{j})}$ or $\frac{1}{d_{ij}}$.

The reciprocal distance matrix $RD$ (or called the Harary matrix) \cite{pntm1993} is defined as
\begin{equation*}
(RD)_{ij}=
       \left\{
        \begin{array}{lc}
    \frac{1}{d(v_{i},v_{j})},\quad  \ i\neq j,\\
    0,\quad \ \  otherwise.
           \end{array}
        \right.
	\end{equation*}

Denote by diagonal matrix $RT(G)$ the vertex reciprocal transmission matrix whose $(i,i)$ element is $RTr_{G}(v_{i})$. Then the reciprocal distance signless Laplacian matrix is defined as \cite{abra2019}
$$RQ(G)=RT(G)+RD(G).$$

To achieve the convex linear combination of adjacent matrix $A(G)$ and degree diagonal matrix $D(G)$, Nikiforov proposed the so called $A_{\alpha}$ matrix \cite{niki2017}, which is defined as
$$A_{\alpha}(G)=\alpha D(G)+(1-\alpha)A(G),\quad 0\leq \alpha \leq 1.$$
For the recent research of $A_{\alpha}$ matrix, one can see \cite{fewe2022,lilf2022,lisu2021,liwa2022,lizh2022,yush2022}.

To achieve the convex linear combination of vertex reciprocal transmission matrix $RT(G)$ and reciprocal distance matrix $RD(G)$, Tian et al. proposed the so called $RD_{\alpha}$ matrix \cite{tccu2022}, which is defined as
$$RD_{\alpha}(G)=\alpha RT(G)+(1-\alpha)RD(G),\quad 0\leq \alpha \leq 1.$$
It is obvious that $RD_{0}(G)=RD(G)$, $RD_{1}(G)=RT(G)$ and $2RD_{\frac{1}{2}}(G)=RQ(G)$.
For the recent research of $D(G)$, $Q(G)$, $RD(G)$, $RQ(G)$ one can see \cite{linh2015,liws2021,lisx2021,pgab2020,suli2014,wawe2022,yory2017}.

In this paper, we use $\lambda(M)$ to denote the eigenvalue of matrix $M$, $\sigma(M)$ to denote the spectrum of matrix $M$, $S(M)$ to denote the spread of matrix $M$.

Let $\lambda_{1}(RD_{\alpha}(G))\geq \lambda_{2}(RD_{\alpha}(G))\geq \cdots \geq \lambda_{n}(RD_{\alpha}(G))$ be the eigenvalues of $RD_{\alpha}$ matrix of graphs $G$. Then the $RD_{\alpha}$-spread of graph $G$ can be defined as $S_{RD_{\alpha}}(G)=\lambda_{1}(RD_{\alpha}(G))-\lambda_{n}(RD_{\alpha}(G))$.
Since $RD_{\alpha}(G)$ is a non-negative and irreducible matrix, thus $\lambda_{1}(RD_{\alpha}(G))$ is a simple (with multiplicity 1) eigenvalue.

In this paper, we first obtain some sharp lower and upper bounds for the $RD_{\alpha}$-spread of graphs. Then we determine the lower bounds for the $RD_{\alpha}$-spread of bipartite graphs and graphs with given clique number.  At last, we determine the $RD_{\alpha}$-spread of double star graphs.
%By the definition of $RD_{\alpha}(G)$, we know
%$$ x^{T}RD_{\alpha}(G)x=\alpha \sum_{i=1}^{n}RTr_{G}(v_{i})x_{i}^{2}+2(1-\alpha)\sum_{1\leq i<j\leq n}\frac{1}{d_{ij}}x_{i}x_{j}, $$
%where $x=(x_{1},x_{2},\cdots,x_{n})^{T}\in R^{n}$.

\section{Main results}
\hskip 0.6cm
The Frobenious norm of matrix $M=(m_{ij})_{n\times n}$ is defined as $\|M\|_{F}=\sqrt{\sum\limits_{i=1}^{n}\sum\limits_{j=1}^{n}|m_{i,j}|^{2}}$.
Let $\sigma(M)=\{\lambda_{1}(M),\lambda_{2}(M),\cdots,\lambda_{n}(M)\}$. Then
$\|M\|_{F}^{2}=\sum\limits_{i=1}^{n}\sum\limits_{j=1}^{n}|m_{i,j}|^{2}=tr(M^{2})=
\sum\limits_{i=1}^{n}|\lambda_{i}(M)|^{2}$.
In particular, $\|RD_{\alpha}\|_{F}^{2}=tr((RD_{\alpha})^{2})=\alpha^{2}\sum\limits_{i=1}^{n}(RTr_{G}(v_{i}))^{2}
+(1-\alpha)^{2}\sum\limits_{i\neq j}(\frac{1}{d_{ij}})^{2}$.

\begin{lemma}\label{l2-1}\cite{mirs1956}
Let $M$ be an n-square normal matrix. Then
$$ S(M)\leq \left( 2\|M\|_{F}^{2}-\frac{2}{n}(tr(M))^{2} \right)^{\frac{1}{2}},$$
with equality if and only if $\lambda_{2}(M)=\lambda_{3}(M)=\cdots=\lambda_{n-1}(M)=\frac{1}{2}(\lambda_{1}(M)+\lambda_{n}(M))$.
\end{lemma}

By Lemma \ref{l2-1} and the definition of $RD_{\alpha}(G)$, we have the following results.
\begin{theorem}\label{t2-2}
Let $G$ be a connected graph with $|V(G)|=n\geq 3$ and Harary index $H(G)$. Then for $\alpha\in [0,1]$, we have
$$ S_{RD_{\alpha}}(G)\leq \left( 2\alpha^{2}\sum_{i=1}^{n}(RTr_{G}(v_{i}))^{2}+2(1-\alpha)^{2}\sum_{i\neq j}(\frac{1}{d_{ij}})^{2}-\frac{8}{n}\alpha^{2}H^{2}(G)  \right)^{\frac{1}{2}}, $$
with equality if and only if $\lambda_{2}(M)=\lambda_{3}(M)=\cdots=\lambda_{n-1}(M)=\frac{1}{2}(\lambda_{1}(M)+\lambda_{n}(M))$.
\end{theorem}

Since $RD_{\alpha}(G)$ is a positive semidefinite matrix if $\alpha\in [\frac{1}{2},1]$ (see Page 9 of \cite{tccu2022}), then $S_{RD_{\alpha}}(G)=\lambda_{1}(RD_{\alpha}(G))-\lambda_{n}(RD_{\alpha}(G))\leq \lambda_{1}(RD_{\alpha}(G))$, with equality if and only if $\lambda_{n}(RD_{\alpha}(G))=0$.

We can easily know $S_{RD_{\alpha}}(G)=(1-\alpha)S_{RD}(G)$ if $G$ is a reciprocal transmission regular graph. Thus, in the following, we suppose $G$ is not reciprocal transmission regular graph.
Similar to the proof of Lemma 4 of \cite{metr2021}, we have
\begin{theorem}\label{t2-3}
Let $G$ be a connected graph with $|V(G)|=n\geq 2$. Then for $\alpha\in [0,1)$, we have
$G$ has exactly two distinct $RD_{\alpha}$ eigenvalues if and only if $G\cong K_{n}$.
And it is easily to know that $\sigma(RD_{\alpha}(K_{n}))=\{n-1,(\alpha n-1)^{[n-1]}\}$ \rm(also see \cite{tccu2022}\rm).
\end{theorem}

\begin{lemma}\label{l2-4}\cite{tccu2022}
Let $G$ be a connected graph with $|V(G)|=n\geq 2$. Then
$$ \sqrt{\frac{\sum\limits_{i=1}^{n}(RTr_{G}(v_{i}))^{2}}{n}}\leq \lambda_{1}(RD_{\alpha}(G))\leq
\max_{1\leq i\leq n} \left( \alpha RTr_{G}(v_{i})+(1-\alpha)\sum_{j=1,j\neq i}^{n}
\frac{1}{d_{G}(v_{i},v_{j})}\sqrt{\frac{RTr_{G}(v_{j})}{RTr_{G}(v_{i})}}    \right),$$
with equality if and only if $G$ is a reciprocal transmission regular graph.
\end{lemma}

\begin{lemma}\label{l2-5}\cite{tccu2022}
Let $G$ be a connected graph with $|V(G)|=n\geq 2$ and Harary index $H(G)$. Then
$$ \lambda_{1}(RD_{\alpha}(G))\geq \frac{2H(G)}{n},$$
with equality if and only if $G$ is a reciprocal transmission regular graph.
\end{lemma}

\begin{theorem}\label{t2-6}
Let $G$ be a connected graph with $|V(G)|=n\geq 2$ and Harary index $H(G)$. Then

$(i)$ for $\alpha\in [0,1)$, we have $S_{RD_{\alpha}}(G)\geq \frac{2(1-\alpha)H(G)}{n-1}$,
with equality if and only if $G\cong K_{n}$.

$(ii)$ for $\alpha\in [\frac{1}{2},1)$, we have $S_{RD_{\alpha}}(G)\geq x-\sqrt{\frac{\|RD_{\alpha}\|_{F}^{2}-x^{2}}{n-1}}$, where $x=\sqrt{\frac{\sum\limits_{i=1}^{n}(RTr_{G}(v_{i}))^{2}}{n}}$ or $\frac{2H(G)}{n}$,
with equality if and only if $G\cong K_{n}$.
\end{theorem}
\begin{proof}
Since $2\alpha H(G)=tr(RD_{\alpha}(G))=\sum\limits_{i=1}^{n}\lambda_{i}\geq \lambda_{1}+(n-1)\lambda_{n}$, then $\theta_{n}\leq \frac{2\alpha H(G)-\theta_{1}}{n-1}$,
thus by Lemma \ref{l2-5}, we have
 $S_{RD_{\alpha}}(G)\geq \lambda_{1}-\frac{2\alpha H(G)-\lambda_{1}}{n-1}=\frac{n}{n-1}\lambda_{1}
 -\frac{2\alpha H(G)}{n-1}\geq \frac{n}{n-1}\frac{2H(G)}{n}-\frac{2\alpha H(G)}{n-1}=\frac{2(1-\alpha)H(G)}{n-1}$, with equality if and only if $\lambda_{2}=\lambda_{3}=
 \cdots \lambda_{n}$ and $G$ is a reciprocal transmission regular graph, i.e., $G$ has two distinct $RD_{\alpha}$ eigenvalues and $G$ is a reciprocal transmission regular graph, then by Theorem \ref{t2-3}, $G\cong K_{n}$.

Note that $RD_{\alpha}(G)$ is a positive semidefinite matrix if $\alpha\in [\frac{1}{2},1)$ (see Page 9 of \cite{tccu2022}).

Since $\|RD_{\alpha}\|_{F}^{2}=tr((RD_{\alpha})^{2})=\sum\limits_{i=1}^{n}|\lambda_{i}(RD_{\alpha})|^{2}
\geq \theta_{1}^{2}(RD_{\alpha})+(n-1)\theta_{n}^{2}(RD_{\alpha})$, with equality if and only if
$RD_{\alpha}(G)$ has two distinct eigenvalues, i.e., $G\cong K_{n}$ (by Theorem \ref{t2-3}).
Then $\lambda_{n}(RD_{\alpha})\leq \sqrt{\frac{\|RD_{\alpha}\|_{F}^{2}-\lambda_{1}^{2}(RD_{\alpha})}{n-1}}$.
Thus $S_{RD_{\alpha}}(G)\geq \lambda_{1}(RD_{\alpha})-\sqrt{\frac{\|RD_{\alpha}\|_{F}^{2}-\lambda_{1}^{2}(RD_{\alpha})}{n-1}}$.
Let $f(x)=x-\sqrt{\frac{\|RD_{\alpha}\|_{F}^{2}-x^{2}}{n-1}}$, $f'(x)=1+\frac{x}{\sqrt{\frac{\|RD_{\alpha}\|_{F}^{2}-x^{2}}{n-1}}}>0$ for any $x\in (0,+\infty)$. Note that $\lambda_{1}(RD_{\alpha})>0$. By Lemma \ref{l2-4} and \ref{l2-5}, we have
$S_{RD_{\alpha}}(G)\geq x-\sqrt{\frac{\|RD_{\alpha}\|_{F}^{2}-x^{2}}{n-1}}$, where $x=\sqrt{\frac{\sum\limits_{i=1}^{n}(RTr_{G}(v_{i}))^{2}}{n}}$ or $\frac{2H(G)}{n}$,
with equality if and only if $G\cong K_{n}$.
\end{proof}

%The following Lemma is the well known intersection theorem.
%Denote by $\lambda_{1}(A)\geq \lambda_{2}(A)\geq \cdots \geq \lambda_{n}(A)$ the eigenvalues of matrix $A$ with order $n$.
%\begin{lemma}\label{l2-7}\cite{brha2010}
%Let $A\in M_{n}(R)$ and $B_{m\times m}$ is the principal submatrix of $A$.
%Then for any $1\leq i\leq m$, we have
%$\lambda_{n-m+i}(A)\leq \lambda_{i}(B)\leq \lambda_{i}(A)$.
%\end{lemma}

Denote by $\lambda_{1}(M)\geq \lambda_{2}(M)\geq \cdots \geq \lambda_{n}(M)$ the eigenvalues of matrix $M$ of order $n$.
\begin{lemma}\label{l2-8}\cite{brha2010}
Let $C=A+B$, where $A, B$ are Hermitian matrices of order $n$.

$(i)$ if $n\geq j\geq i\geq 1$, then
\begin{equation}\label{eq:21}
 \lambda_{i}(C)\geq \lambda_{j}(A)+\lambda_{i-j+n}(B),
\end{equation}

$(ii)$ if $n\geq i\geq j\geq 1$, then
\begin{equation}\label{eq:22}
 \lambda_{i}(C)\leq \lambda_{j}(A)+\lambda_{i-j+1}(B),
\end{equation}
with both equalities if and only if there is a nonzero vector that is an eigenvector to each of the three eigenvalues involved.
\end{lemma}

By Lemma \ref{l2-8}, Tian et al.,\cite{tccu2022} obtained the following bounds for $\lambda_{k}(RD_{\alpha}(G))$ in terms of $\lambda_{k}(RD(G))$.

\begin{lemma}\label{l2-9}\cite{tccu2022}
Let $G$ be a connected graph with $|V(G)|=n$. Let $RTr_{1}$ (resp. $RTr_{n}$) the maximum reciprocal transmission (resp. minimum reciprocal transmission). Then
$$ \alpha RTr_{n}+(1-\alpha)\lambda_{k}(RD(G))\leq \lambda_{k}(RD_{\alpha}(G))\leq \alpha RTr_{1}+(1-\alpha)\lambda_{k}(RD(G)).  $$
\end{lemma}
In fact, we can easily know the bounds with equality if and only if $G$ is reciprocal transmission regular. Further, we have

\begin{theorem}\label{t2-10}
Let $G$ be a connected graph with $|V(G)|=n$. $RTr_{1}$ (resp. $RTr_{n}$) defined as Lemma \ref{l2-9}. Then
$$ \alpha(RTr_{n}-RTr_{1})+(1-\alpha)S_{RD}(G)\leq S_{RD_{\alpha}}(G)\leq \alpha(RTr_{1}-RTr_{n})+(1-\alpha)S_{RD}(G),  $$
with equality if and only if $G$ is reciprocal transmission regular.
\end{theorem}

Next, we consider the bounds of $S_{RD_{\alpha}}(G)$ of graphs in terms of $S_{A_{\alpha}}(G)$ or $S_{A_{\alpha}}(\overline{G})$.
\begin{theorem}\label{t2-11}
Let $G$ be a connected graph with $|V(G)|=n$ and diameter $d$. Then

$(i)$ if $d=2$, then $S_{RD_{\alpha}}(G)\leq (1-\alpha)n+\frac{1}{2}S_{A_{\alpha}}(\overline{G})$,
with equality if and only if $G$ is reciprocal transmission regular.

$(ii)$ if $d\geq 3$, then $S_{RD_{\alpha}}(G)\leq \frac{1}{2}(1-\alpha)n+\frac{1}{2}S_{A_{\alpha}}(G)+S(M^{*})$,
where $M^{*}(G)=\alpha RTr_{G}^{*}+(1-\alpha)M(G)$,
$M(G)=(t_{ij})_{n\times n}$, $t_{ij}=\min\{0,\frac{1}{d_{ij}}-\frac{1}{2}\}$;
$RTr_{G}^{*}(v_{i})=\sum\limits_{d_{ij}\geq 3}(\frac{1}{d_{ij}}-\frac{1}{2})$,
$RTr_{G}^{*}=diag( RTr_{G}^{*}(v_{1}),RTr_{G}^{*}(v_{2}), \cdots, RTr_{G}^{*}(v_{n}))$
be a diagonal matrix.
\end{theorem}
\begin{proof}
Suppose that $spec(A_{\alpha}(G))=\{\mu_{1},\mu_{2},\cdots,\mu_{n}\}$ and $\mu_{1}\geq \mu_{2}\geq\cdots\geq \mu_{n}$; $spec(A_{\alpha}(\overline{G}))=\{\overline{\mu_{1}},\overline{\mu_{2}},\cdots,\overline{\mu_{n}}\}$ and $\overline{\mu_{1}}\geq \overline{\mu_{2}}\geq\cdots\geq \overline{\mu_{n}}$.
Let $J_{n}$ be the all one matrix, $I_{n}$ the unit matix. Let $D(G)=diag(d_{1},d_{2},\cdots,d_{n})$ is the degree diagonal matrix.

$(i)$ Since $d=2$, then $RD(G)=\frac{1}{2}(J_{n}-I_{n}+A(G))=A(G)+\frac{1}{2}A(\overline{G})$ and
$RTr_{G}(v_{i})=d_{i}+\frac{1}{2}(n-d_{i}-1)=\frac{1}{2}(n-1+d_{i})$. Thus
\begin{eqnarray*}
RD_{\alpha}(G) & = & \alpha RT(G)+(1-\alpha)RD(G) \\
& = & \frac{1}{2}\alpha(n-1)I_{n}+\frac{1}{2}\alpha D(G)+(1-\alpha)(A(G)+\frac{1}{2}A(\overline{G})) \\
& = & \frac{1}{2}\alpha(n-1)I_{n}+\frac{1}{2}\alpha D(G)+\frac{1}{2}\alpha D(\overline{G})+(1-\alpha)(A(G)+A(\overline{G}))\\
& \quad & -\frac{1}{2}(\alpha D(\overline{G})+(1-\alpha)A(\overline{G}))\\
& = & \frac{1}{2}\alpha(n-1)I_{n}+\frac{1}{2}\alpha (n-1)I_{n}+(1-\alpha)(J_{n}-I_{n})-\frac{1}{2}A_{\alpha}(\overline{G})\\
& = & (\alpha n-1)I_{n}-\frac{1}{2}A_{\alpha}(\overline{G})+(1-\alpha)J_{n}.
\end{eqnarray*}
By Lemma \ref{l2-8}, we have
$$ \lambda_{1}(RD_{\alpha}(G))\leq \alpha n-1-\frac{1}{2}\overline{\mu_{n}}+(1-\alpha)n=n-1-\frac{1}{2}\overline{\mu_{n}},$$
$$ \lambda_{n}(RD_{\alpha}(G))\geq \alpha n-1-\frac{1}{2}\overline{\mu_{1}}.$$
Thus $S_{RD_{\alpha}}(G)\leq (1-\alpha)n+\frac{1}{2}S_{A_{\alpha}}(\overline{G})$,
with equality if and only if $G$ is reciprocal transmission regular.

$(ii)$ Since $d\geq 3$, then
$RTr_{G}(v_{i})=d_{i}\times 1+\frac{1}{2}\times \overline{d_{i}}+RTr^{*}_{G}(v_{i})$,
where $RTr_{G}^{*}(v_{i})=\sum\limits_{d_{ij}\geq 3}(\frac{1}{d_{ij}}-\frac{1}{2})$,
$RTr_{G}^{*}=diag( RTr_{G}^{*}(v_{1}),RTr_{G}^{*}(v_{2}), \cdots, RTr_{G}^{*}(v_{n}))$
be a diagonal matrix. Let $M^{*}(G)=\alpha RTr_{G}^{*}+(1-\alpha)M(G)$,
$M(G)=(t_{ij})_{n\times n}$, $t_{ij}=\min\{0,\frac{1}{d_{ij}}-\frac{1}{2}\}$.
Thus
\begin{eqnarray*}
RD_{\alpha}(G) & = & \alpha RT(G)+(1-\alpha)RD(G) \\
& = & \alpha \left( D(G)+\frac{1}{2}D(\overline{G})+RTr_{G}^{*} \right) +(1-\alpha)\left( A(G)+\frac{1}{2}A(\overline{G})+M(G)\right) \\
& = & \frac{1}{2}\alpha (D(G)+D(\overline{G})) +\frac{1}{2}(1-\alpha) (A(G)+A(\overline{G}))+\frac{1}{2}(\alpha D(G)+(1-\alpha)A(G))\\
& \quad & +\alpha RTr_{G}^{*}+(1-\alpha)M(G)\\
& = & \frac{1}{2}(\alpha D(K_{n})+(1-\alpha)A(K_{n}))+\frac{1}{2}A_{\alpha}(G)+\alpha RTr_{G}^{*}+(1-\alpha)M(G)\\
& = & \frac{1}{2}A_{\alpha}(K_{n})+\frac{1}{2}A_{\alpha}(G)+M^{*}(G).
\end{eqnarray*}
Note that $spec(A_{\alpha}(G))=(n-1,(\alpha n-1)^{[n-1]})$. By Lemma \ref{l2-8}, we have
\begin{eqnarray*}
\lambda_{1}(RD_{\alpha}(G)) & \leq &  \lambda_{1}(\frac{1}{2}A_{\alpha}(K_{n}))+\lambda_{1}(\frac{1}{2}A_{\alpha}(G)+M^{*}(G))\\
& = & \frac{1}{2}(n-1)+\lambda_{1}(\frac{1}{2}A_{\alpha}(G)+M^{*}(G)) \\
& \leq & \frac{1}{2}(n-1)+\frac{1}{2}\mu_{1}+\lambda_{1}(M^{*}(G)).
\end{eqnarray*}
\begin{eqnarray*}
\lambda_{n}(RD_{\alpha}(G)) & \leq &  \lambda_{n}(\frac{1}{2}A_{\alpha}(K_{n}))+\lambda_{n}(\frac{1}{2}A_{\alpha}(G)+M^{*}(G))\\
& = & \frac{1}{2}(\alpha n-1)+\lambda_{n}(\frac{1}{2}A_{\alpha}(G)+M^{*}(G)) \\
& \geq & \frac{1}{2}(\alpha n-1)+\frac{1}{2}\mu_{n}+\lambda_{1}(M^{*}(G)).
\end{eqnarray*}
Thus $S_{RD_{\alpha}}(G)\leq \frac{1}{2}(n-1)+\frac{1}{2}\mu_{1}+\lambda_{1}(M^{*}(G))-(\frac{1}{2}(\alpha n-1)+\frac{1}{2}\mu_{n}+\lambda_{1}(M^{*}(G)))=\frac{1}{2}(1-\alpha)n
+\frac{1}{2}S_{A_{\alpha}}(G)+S(M^{*})$.
\end{proof}

\begin{lemma}\label{l2-12}\cite{brha2010}
Let $A$ be a real symmetric matrix and $B$ is the quotient matrix of $A$.
Then the eigenvalues of $B$ interlace the eigenvalues of $A$.
Moreover if $B$ is an equitable quotient matrix, then the spectrum of $B$ is
contained in the spectrum of $A$.
\end{lemma}

\begin{lemma}\label{l2-13}\cite{tccu2022}
For $a,b\geq 1$, let $K_{a,b}$ be a complete bipartite graph with $a+b=n$ vertices.
Then the spectrum of $RD_{\alpha}(K_{a,b})$ consists of
$$\left(\frac{\alpha(n+b)-1}{2}\right)^{[a-1]},\left(\frac{\alpha(n+a)-1}{2}\right)^{[b-1]} ,$$
and the rest eigenvalues are
$$ \frac{1}{2}\left( (\alpha+\frac{1}{2})n-1\pm \sqrt{(\alpha-\frac{1}{2})^{2}(a-b)^{2}+4(1-\alpha)^{2}ab}  \right)  .$$
\end{lemma}

If $G$ is a bipartite graph and $\Delta=n-1$, then $G\cong K_{1,n-1}$.
The spread of $K_{1,n-1}$ can be determined by Lemma \ref{l2-13}, thus
in the following we only consider $\Delta\leq n-2$ for bipartite graphs.

For convenience, we let $t_{v_{i}}=\frac{1}{d_{G}(v_{i})}\sum\limits_{v_{i}v_{j}\in E(G)}RTr_{G}(v_{j})$.

\begin{theorem}\label{t2-14}
Let $G$ be a bipartite graph with $|V(G)|=n\geq 3$, maximum degree $\Delta \leq n-2$, Harary index $H(G)$. Suppose that for $v_{i}\in V(G)$, we have $d_{G}(v_{i})=\Delta$ for $i=1,2,\cdots,k$. Then
$$ S_{RD_{\alpha}}(G)\geq \max_{1\leq i\leq k} \sqrt{(b_{11}+b_{22})^{2}-4(b_{11}b_{22}-b_{12}b_{21})}, $$
where
$$b_{11}=\frac{1}{\Delta+1} \left( \frac{1}{2}(1-\alpha)\Delta (\Delta+3)+\alpha \Delta t_{v_{i}}+\alpha RTr_{G}(v_{i})  \right),$$
$$b_{12}=\frac{1-\alpha}{\Delta+1} \left( \Delta t_{v_{i}}+RTr_{G}(v_{i})-\frac{1}{2}\Delta(\Delta+3)  \right),$$
$$b_{21}=\frac{1-\alpha}{n-\Delta-1} \left( \Delta t_{v_{i}}+RTr_{G}(v_{i})-\frac{1}{2}\Delta(\Delta+3)  \right),$$
$$b_{22}=\frac{1}{n-\Delta-1} \left( \alpha(2H-\Delta t_{v_{i}}-RTr_{G}(v_{i}) )+(1-\alpha)(2H-2\Delta t_{v_{i}}-2RTr_{G}(v_{i})+\frac{1}{2}\Delta(\Delta+3) )  \right).$$
\end{theorem}
\begin{proof}
Suppose that $N(v_{i})=\{u_{i_{1}},u_{i_{2}},\cdots,u_{i_{\Delta}}\}$ $(i=1,2,\cdots,k)$ and
$N[v_{i}]=N(v_{i})\bigcup \{v_{i}\}$.
Let $V(G)=\{v_{i},u_{i_{1}},u_{i_{2}},\cdots,u_{i_{\Delta}},w_{1},w_{2},\cdots,w_{n-\Delta-1}\}$.
Then
\begin{equation}
RD_{\alpha}(G)= \left(
   \begin{array}{cc}
    X & (1-\alpha)Y \\
    (1-\alpha)Y^{T} & Z   \notag
   \end{array}
\right),
\end{equation}
where
\begin{equation}
X= \left(
   \begin{array}{ccccc}
    \alpha RTr(v_{i}) & 1-\alpha & 1-\alpha & \cdots & 1-\alpha \\
    1-\alpha & \alpha RTr(u_{i_{1}}) & \frac{1-\alpha}{2} & \cdots & \frac{1-\alpha}{2} \\
    1-\alpha & \frac{1-\alpha}{2} & \alpha RTr(u_{i_{2}}) & \cdots & \frac{1-\alpha}{2} \\
    \vdots & \vdots & \vdots & \ddots & \vdots \\
    1-\alpha & \frac{1-\alpha}{2} & \frac{1-\alpha}{2} & \cdots & \alpha RTr(u_{i_{\Delta}})    \notag
   \end{array}
\right)
\end{equation}

\begin{equation}
Y= \left(
   \begin{array}{ccccc}
    \frac{1}{d(v_{i},w_{1})} & \frac{1}{d(v_{i},w_{2})} & \frac{1}{d(v_{i},w_{3})} & \cdots & \frac{1}{d(v_{i},w_{n-\Delta-1})} \\
    \frac{1}{d(u_{i_{1}},w_{1})} & \frac{1}{d(u_{i_{1}},w_{2})} & \frac{1}{d(u_{i_{1}},w_{3})} & \cdots & \frac{1}{d(u_{i_{1}},w_{n-\Delta-1})} \\
    \frac{1}{d(u_{i_{2}},w_{1})} & \frac{1}{d(u_{i_{2}},w_{2})} & \frac{1}{d(u_{i_{2}},w_{3})} & \cdots & \frac{1}{d(u_{i_{2}},w_{n-\Delta-1})} \\
    \vdots & \vdots & \vdots & \ddots & \vdots \\
    \frac{1}{d(u_{i_{\Delta}},w_{1})} & \frac{1}{d(u_{i_{\Delta}},w_{2})} & \frac{1}{d(u_{i_{\Delta}},w_{3})} & \cdots & \frac{1}{d(u_{i_{\Delta}},w_{n-\Delta-1})}    \notag
   \end{array}
\right)
\end{equation}

\begin{equation}
Z= \left(
   \begin{array}{ccccc}
    \alpha RTr(w_{1}) & \frac{1-\alpha}{d(w_{1},w_{2})} & \frac{1-\alpha}{d(w_{1},w_{3})} & \cdots & \frac{1-\alpha}{d(w_{1},w_{n-\Delta-1})} \\
    \frac{1-\alpha}{d(w_{2},w_{1})} & \alpha RTr(w_{2}) & \frac{1-\alpha}{d(w_{2},w_{3})} & \cdots & \frac{1-\alpha}{d(w_{2},w_{n-\Delta-1})} \\
    \frac{1-\alpha}{d(w_{3},w_{1})} & \frac{1-\alpha}{d(w_{3},w_{2})} & \alpha RTr(w_{3}) & \cdots & \frac{1-\alpha}{d(w_{3},w_{n-\Delta-1})} \\
    \vdots & \vdots & \vdots & \ddots & \vdots \\
    \frac{1-\alpha}{d(w_{n-\Delta-1},w_{1})} & \frac{1-\alpha}{d(w_{n-\Delta-1},w_{2})} & \frac{1-\alpha}{d(w_{n-\Delta-1},w_{3})} & \cdots & \alpha RTr(w_{n-\Delta-1})    \notag
   \end{array}
\right)
\end{equation}
Then the quotient matrix $B$ of matrix $RD_{\alpha}(G)$ is
\begin{equation}
B= \left(
   \begin{array}{cc}
    b_{11} & b_{12} \\
    b_{21} & b_{22}   \notag
   \end{array}
\right),
\end{equation}
where
$$b_{11}=\frac{1}{\Delta+1} \left( \frac{1}{2}(1-\alpha)\Delta (\Delta+3)+\alpha \Delta t_{v_{i}}+\alpha RTr_{G}(v_{i})  \right),$$
$$b_{12}=\frac{1-\alpha}{\Delta+1} \left( \Delta t_{v_{i}}+RTr_{G}(v_{i})-\frac{1}{2}\Delta(\Delta+3)  \right),$$
$$b_{21}=\frac{1-\alpha}{n-\Delta-1} \left( \Delta t_{v_{i}}+RTr_{G}(v_{i})-\frac{1}{2}\Delta(\Delta+3)  \right),$$
$$b_{22}=\frac{1}{n-\Delta-1} \left( \alpha(2H-\Delta t_{v_{i}}-RTr_{G}(v_{i}) )+(1-\alpha)(2H-2\Delta t_{v_{i}}-2RTr_{G}(v_{i})+\frac{1}{2}\Delta(\Delta+3) )  \right).$$
The eigenvalues of the matrix $B$ are
$$ \frac{1}{2}\left( b_{11}+b_{22}\pm \sqrt{(b_{11}+b_{22})^{2}-4(b_{11}b_{22}-b_{12}b_{21})} \right)   .$$
By Lemma \ref{l2-12}, we have
$$ S_{RD_{\alpha}}(G)\geq \max_{1\leq i\leq k} \sqrt{(b_{11}+b_{22})^{2}-4(b_{11}b_{22}-b_{12}b_{21})}.$$
\end{proof}

A clique of graph $G$ is a vertex subset $V_{0}\subseteq V(G)$ such that in $G[V_{0}]$,
the subgraph induced by $V_{0}$, any two vertices are adjacent.
Denote by $\omega(G)$ the clique number, which is the number of vertices in a largest clique of $G$.

Next, we consider the lower bound of $ S_{RD_{\alpha}}(G)$ with given clique number $\omega(G)$.
If $\omega=n$, then $G\cong K_{n}$. One know $S_{RD_{\alpha}}(K_{n})=n(1-\alpha)$. Thus
in the following we only consider $\omega\leq n-1$.

\begin{theorem}\label{t2-15}
Let $G$ be a connected graph with $|V(G)|=n\geq 3$, clique number $2\leq \omega\leq n-1$, Harary index $H(G)$. Suppose that $G_{1},G_{2},\cdots, G_{k}$ are all cliques of $G$ with $\omega$ vertices and  $s_{i}=\sum\limits_{v_{j}\in E(G_{i})}RTr(v_{j})$. Then
$$ S_{RD_{\alpha}}(G)\geq \max_{1\leq i\leq k} \sqrt{(c_{11}+c_{22})^{2}-4(c_{11}c_{22}-c_{12}c_{21})}, $$
where
$$c_{11}=\frac{1}{\omega} \left( \alpha s_{i}+(1-\alpha)\omega(\omega-1)  \right),$$
$$c_{12}=\frac{1-\alpha}{\omega} \left( s_{i}- \omega(\omega-1) \right),$$
$$c_{21}=\frac{1-\alpha}{n-\omega} \left( s_{i}- \omega(\omega-1) \right),$$
$$c_{22}=\frac{1}{n-\omega} \left( \alpha(2H-s_{i})+(1-\alpha)(2H-2s_{i}+\omega(\omega-1))  \right).$$
\end{theorem}
\begin{proof}
Suppose that $V(G_{i})=\{u_{i_{1}},u_{i_{2}},\cdots,u_{i_{\omega}}\}$ $(i=1,2,\cdots,k)$ and

$V(G)=\{w_{1},w_{2},\cdots,w_{n-\omega}\}\bigcup V(G_{i})$.
Then
\begin{equation}
RD_{\alpha}(G)= \left(
   \begin{array}{cc}
    X & (1-\alpha)Y \\
    (1-\alpha)Y^{T} & Z   \notag
   \end{array}
\right),
\end{equation}

where
\begin{equation}
X= \left(
   \begin{array}{ccccc}
    \alpha RTr(u_{i_{1}}) & 1-\alpha & 1-\alpha & \cdots & 1-\alpha \\
    1-\alpha & \alpha RTr(u_{i_{2}}) & 1-\alpha & \cdots & 1-\alpha \\
    1-\alpha & 1-\alpha & \alpha RTr(u_{i_{3}}) & \cdots & 1-\alpha \\
    \vdots & \vdots & \vdots & \ddots & \vdots \\
    1-\alpha & 1-\alpha & 1-\alpha & \cdots & \alpha RTr(u_{i_{\omega}})    \notag
   \end{array}
\right)
\end{equation}

\begin{equation}
Y= \left(
   \begin{array}{ccccc}
    \frac{1}{d(u_{i_{1}},v_{1})} & \frac{1}{d(u_{i_{1}},v_{2})} & \frac{1}{d(u_{i_{1}},v_{3})} & \cdots & \frac{1}{d(u_{i_{1}},v_{n-\omega})} \\
    \frac{1}{d(u_{i_{2}},v_{1})} & \frac{1}{d(u_{i_{2}},v_{2})} & \frac{1}{d(u_{i_{2}},v_{3})} & \cdots & \frac{1}{d(u_{i_{2}},v_{n-\omega})} \\
    \frac{1}{d(u_{i_{3}},v_{1})} & \frac{1}{d(u_{i_{3}},v_{2})} & \frac{1}{d(u_{i_{3}},v_{3})} & \cdots & \frac{1}{d(u_{i_{3}},v_{n-\omega})} \\
    \vdots & \vdots & \vdots & \ddots & \vdots \\
    \frac{1}{d(u_{i_{\omega}},v_{1})} & \frac{1}{d(u_{i_{\omega}},v_{2})} & \frac{1}{d(u_{i_{\omega}},v_{3})} & \cdots & \frac{1}{d(u_{i_{\omega}},v_{n-\omega})}    \notag
   \end{array}
\right)
\end{equation}

\begin{equation}
Z= \left(
   \begin{array}{ccccc}
    \alpha RTr(v_{1}) & \frac{1-\alpha}{d(v_{1},v_{2})} & \frac{1-\alpha}{d(v_{1},v_{3})} & \cdots & \frac{1-\alpha}{d(v_{1},v_{n-\omega})} \\
    \frac{1-\alpha}{d(v_{2},v_{1})} & \alpha RTr(v_{2}) & \frac{1-\alpha}{d(v_{2},v_{3})} & \cdots & \frac{1-\alpha}{d(v_{2},v_{n-\omega})} \\
    \frac{1-\alpha}{d(v_{3},v_{1})} & \frac{1-\alpha}{d(v_{3},v_{2})} & \alpha RTr(v_{3}) & \cdots & \frac{1-\alpha}{d(v_{3},v_{n-\omega})} \\
    \vdots & \vdots & \vdots & \ddots & \vdots \\
    \frac{1-\alpha}{d(v_{n-\omega},v_{1})} & \frac{1-\alpha}{d(v_{n-\omega},v_{2})} & \frac{1-\alpha}{d(v_{n-\omega},v_{3})} & \cdots & \alpha RTr(v_{n-\omega})    \notag
   \end{array}
\right)
\end{equation}

Then the quotient matrix $B$ of matrix $RD_{\alpha}(G)$ is
\begin{equation}
B= \left(
   \begin{array}{cc}
    c_{11} & c_{12} \\
    c_{21} & c_{22}   \notag
   \end{array}
\right),
\end{equation}

where
$$c_{11}=\frac{1}{\omega} \left( \alpha s_{i}+(1-\alpha)\omega(\omega-1)  \right),$$
$$c_{12}=\frac{1-\alpha}{\omega} \left( s_{i}- \omega(\omega-1) \right),$$
$$c_{21}=\frac{1-\alpha}{n-\omega} \left( s_{i}- \omega(\omega-1) \right),$$
$$c_{22}=\frac{1}{n-\omega} \left( \alpha(2H-s_{i})+(1-\alpha)(2H-2s_{i}+\omega(\omega-1))  \right).$$

The eigenvalues of the matrix $B$ are
$$ \frac{1}{2}\left( c_{11}+c_{22}\pm \sqrt{(c_{11}+c_{22})^{2}-4(c_{11}c_{22}-c_{12}c_{21})} \right)   .$$

By Lemma \ref{l2-12}, we have
$$ S_{RD_{\alpha}}(G)\geq \max_{1\leq i\leq k} \sqrt{(c_{11}+c_{22})^{2}-4(c_{11}c_{22}-c_{12}c_{21})}.$$
\end{proof}

In \cite{frtr2016}, a way of matrix decomposition was introduced.
Let $M$ be a $n\times n$ symmetric matrix given in (\ref{eq:216}) where block $E\in R^{t\times t}$, block $\gamma\in R^{t\times s}$, block $F\in R^{s\times s}$, block $Q\in R^{s\times s}$ and $n=zs+t$, $z$ denotes number of copies of $F$. Denote by $\sigma(X)$ spectrum of $X$, and $\sigma^{k}(X)$ multiset with $k$ copies of $\sigma(X)$.
\begin{equation}\label{eq:216}
M= \left(
   \begin{array}{ccccc}
    E & \gamma & \gamma & \cdots & \gamma \\
    \gamma^{T} & F & Q & \cdots & Q \\
    \gamma^{T} & Q & F & \cdots & Q \\
    \vdots & \vdots & \vdots & \ddots & \vdots \\
    \gamma^{T} & Q & Q & \cdots & F
   \end{array}
\right)
\end{equation}

\begin{lemma}\label{l2-16}\cite{frtr2016}
Let $M$ be the matrix in $(\ref{eq:216})$. Then

$(i)$ The spectrum $\sigma(F-Q)\subseteq \sigma(M)$ with multiplicity $z-1$, where $z$ denotes number of copies of the block in matrix $M$.

$(ii)$ The spectrum $\sigma(M)\setminus \sigma^{z-1}(F-Q)=\sigma(M^{'})$ is the collection of remaining $s+t$ eigenvalues of $M$, and
\begin{equation}
M^{'}= \left(
   \begin{array}{cc}
    E & \sqrt{z}\gamma \\
   \sqrt{z}\gamma^{T} & F+(z-1)Q   \notag
   \end{array}
\right).
\end{equation}
\end{lemma}

By Lemma \ref{l2-16}, we have that $\sigma(M)=\sigma^{z-1}(F-Q)\bigcup \sigma(M^{'})$.

Denote by $S_{m,n}$ the double star tree with $m+n+2$ vertices, which is obtained from
star $K_{1,m}$ and $K_{1,n}$ by connecting their root vertices.
In the following we determined the $RD_{\alpha}$-spectrum of $S_{m,n}$ by Lemma \ref{l2-16}.

\begin{theorem}\label{t2-17}
The $RD_{\alpha}$-spectrum of the double star tree $S_{m,n}$ is
$$ \{ \left( \alpha(\frac{1}{2}n+\frac{1}{3}m+1)-\frac{1}{2}(1-\alpha) \right)^{[m-1]},
\left( \alpha(\frac{1}{2}m+\frac{1}{3}n+1)-\frac{1}{2}(1-\alpha) \right)^{[n-1]}, \theta_{1}, \theta_{2}, \theta_{3}, \theta_{4}   \},   $$
where $\theta_{i}$ $(i=1,2,3,4)$ are the eigenvalues of $U^{**}$, and
\begin{equation}
U^{**}= \left(
   \begin{array}{cccc}
    \alpha(m+\frac{1}{2}n+1) & 1-\alpha & \frac{1}{2}(1-\alpha)\sqrt{n} & (1-\alpha)\sqrt{m}\\
    1-\alpha & \alpha(n+\frac{1}{2}m+1) & (1-\alpha)\sqrt{n} & \frac{1}{2}(1-\alpha)\sqrt{m}\\
    \frac{1}{2}(1-\alpha)\sqrt{n} & (1-\alpha)\sqrt{n} & l_{1} & \frac{1}{3}(1-\alpha)\sqrt{mn}\\
    (1-\alpha)\sqrt{m} & \frac{1}{2}(1-\alpha)\sqrt{m} & \frac{1}{3}(1-\alpha)\sqrt{mn} & l_{2}   \notag
   \end{array}
\right).
\end{equation}
$l_{1}=\alpha(\frac{1}{2}n+\frac{1}{3}m+1)+\frac{1}{2}(1-\alpha)(n-1)$,
$l_{2}=\alpha(\frac{1}{2}m+\frac{1}{3}n+1)+\frac{1}{2}(1-\alpha)(m-1)$.
\end{theorem}
\begin{proof}
For the double star tree $S_{m,n}$, let $V(K_{1,m})=\{u,x_{1},x_{2},\cdots,x_{m}\}$ and
$V(K_{1,n})=\{v,y_{1},y_{2},\cdots,y_{n}\}$, where $u$ (resp., $v$) is the root vertices of $V(K_{1,m})$ (resp., $V(K_{1,n})$). Then $V(S_{m,n})=\{u,v, x_{1},x_{2},\cdots,x_{m},y_{1},y_{2},\cdots,y_{n}\}$.

It is obvious that $RTr_{G}(u)=m+\frac{1}{2}n+1$, $RTr_{G}(v)=n+\frac{1}{2}m+1$,
$RTr_{G}(x_{i})=\frac{1}{2}m+\frac{1}{3}n+1$ for $i=1,2,\cdots,m$ and $RTr_{G}(y_{i})=\frac{1}{2}n+\frac{1}{3}m+1$ for $i=1,2,\cdots,n$.
We research matrix in the order of $\{u,v,x_{1},x_{2},\cdots,x_{m},y_{1},y_{2},\cdots,y_{n}\}$.

The $RD_{\alpha}$ matrix of $S_{m,n}$ can be written as
\begin{equation}
RD_{\alpha}(S_{m,n})= \left(
   \begin{array}{ccccc}
    E & \gamma & \gamma & \cdots & \gamma \\
    \gamma^{T} & \alpha(\frac{1}{2}n+\frac{1}{3}m+1) & \frac{1}{2}(1-\alpha) & \cdots & \frac{1}{2}(1-\alpha) \\
    \gamma^{T} & \frac{1}{2}(1-\alpha) & \alpha(\frac{1}{2}n+\frac{1}{3}m+1) & \cdots & \frac{1}{2}(1-\alpha) \\
    \vdots & \vdots & \vdots & \ddots & \vdots \\
    \gamma^{T} & \frac{1}{2}(1-\alpha) & \frac{1}{2}(1-\alpha) & \cdots & \alpha(\frac{1}{2}n+\frac{1}{3}m+1)     \notag
   \end{array}
\right),
\end{equation}

where
\begin{equation}
E= \left(
   \begin{array}{ccccc}
    \alpha(m+\frac{1}{2}n+1) & 1-\alpha & 1-\alpha & \cdots & 1-\alpha \\
    1-\alpha & \alpha(n+\frac{1}{2}m+1) & \frac{1}{2}(1-\alpha) & \cdots & \frac{1}{2}(1-\alpha) \\
    1-\alpha & \frac{1}{2}(1-\alpha) & \alpha(\frac{1}{2}m+\frac{1}{3}n+1) & \cdots & \frac{1}{2}(1-\alpha) \\
    \vdots & \vdots & \vdots & \ddots & \vdots \\
    1-\alpha & \frac{1}{2}(1-\alpha) & \frac{1}{2}(1-\alpha) & \cdots & \alpha(\frac{1}{2}m+\frac{1}{3}n+1)     \notag
   \end{array}
\right)
\end{equation}

and $$\gamma=\left( \frac{1}{2}(1-\alpha),1-\alpha,\frac{1}{3}(1-\alpha),\cdots, \frac{1}{3}(1-\alpha) \right)^{T}.$$

Since $F=[\alpha(\frac{1}{2}n+\frac{1}{3}m+1)]$, $Q=[\frac{1}{2}(1-\alpha)]$ and $z=n$, by Lemma \ref{l2-16}, we have

$\sigma(RD_{\alpha}(S_{m,n}))=\sigma^{n-1}(F-Q)\bigcup \sigma(M^{'})=\sigma^{n-1}([\alpha(\frac{1}{2}n+\frac{1}{3}m+1)-\frac{1}{2}(1-\alpha)])\bigcup \sigma(M^{'})$, where
\begin{equation}
M^{'}= \left(
   \begin{array}{cc}
    E & \sqrt{n}\gamma \\
   \sqrt{n}\gamma^{T} & \alpha(\frac{1}{2}n+\frac{1}{3}m+1)+\frac{1}{2}(m-1)(n-1)   \notag
   \end{array}
\right)_{(m+3)\times (m+3) }.
\end{equation}

First we swap the third column with the last column of the matrix $M^{'}$,
then we swap the third row with the last row.
We can obtain a matrix $U^{*}\sim M^{'}$, i.e., $U^{*}$ and $M^{'}$ are similar, where

\begin{equation}
U^{*}= \left(
   \begin{array}{ccccc}
    E^{*} & \gamma^{*} & \gamma^{*} & \cdots & \gamma^{*} \\
    (\gamma^{*})^{T} & F^{*} & Q^{*} & \cdots & Q^{*} \\
    (\gamma^{*})^{T} & Q^{*} & F^{*} & \cdots & Q^{*} \\
    \vdots & \vdots & \vdots & \ddots & \vdots \\
    (\gamma^{*})^{T} & Q^{*} & Q^{*} & \cdots & F^{*}     \notag
   \end{array}
\right),
\end{equation}

where
\begin{equation}
E^{*}= \left(
   \begin{array}{ccc}
    \alpha(m+\frac{1}{2}n+1) & 1-\alpha & \frac{1}{2}(1-\alpha)\sqrt{n}  \\
    1-\alpha & \alpha(n+\frac{1}{2}m+1) & (1-\alpha)\sqrt{n}  \\
    \frac{1}{2}(1-\alpha)\sqrt{n} & (1-\alpha)\sqrt{n} & \alpha(\frac{1}{2}n+\frac{1}{3}m+1)+\frac{1}{2}(1-\alpha)(n-1)      \notag
   \end{array}
\right)
\end{equation}

and $$\gamma^{*}=\left( 1-\alpha,\frac{1}{2}(1-\alpha),\frac{1}{3}(1-\alpha)\sqrt{n} \right)^{T}.$$

Since $F^{*}=[\alpha(\frac{1}{2}m+\frac{1}{3}n+1)]$, $Q^{*}=[\frac{1}{2}(1-\alpha)]$ and $z^{*}=m$, by Lemma \ref{l2-16}, we have

$\sigma(M^{'})=\sigma^{m-1}(F^{*}-Q^{*})\bigcup \sigma(U^{**})=\sigma^{m-1}([\alpha(\frac{1}{2}m+\frac{1}{3}n+1)-\frac{1}{2}(1-\alpha)])\bigcup \sigma(U^{**})$, where
\begin{equation}
M^{'}= \left(
   \begin{array}{cc}
    E & \sqrt{n}\gamma \\
   \sqrt{n}\gamma^{T} & \alpha(\frac{1}{2}n+\frac{1}{3}m+1)+\frac{1}{2}(m-1)(n-1)   \notag
   \end{array}
\right)_{(m+3)\times (m+3) }.
\end{equation}

\begin{equation}
U^{**}= \left(
   \begin{array}{cccc}
    \alpha(m+\frac{1}{2}n+1) & 1-\alpha & \frac{1}{2}(1-\alpha)\sqrt{n} & (1-\alpha)\sqrt{m}\\
    1-\alpha & \alpha(n+\frac{1}{2}m+1) & (1-\alpha)\sqrt{n} & \frac{1}{2}(1-\alpha)\sqrt{m}\\
    \frac{1}{2}(1-\alpha)\sqrt{n} & (1-\alpha)\sqrt{n} & l_{1} & \frac{1}{3}(1-\alpha)\sqrt{mn}\\
    (1-\alpha)\sqrt{m} & \frac{1}{2}(1-\alpha)\sqrt{m} & \frac{1}{3}(1-\alpha)\sqrt{mn} & l_{2}   \notag
   \end{array}
\right).
\end{equation}
$l_{1}=\alpha(\frac{1}{2}n+\frac{1}{3}m+1)+\frac{1}{2}(1-\alpha)(n-1)$,
$l_{2}=\alpha(\frac{1}{2}m+\frac{1}{3}n+1)+\frac{1}{2}(1-\alpha)(m-1)$.

Thus the $RD_{\alpha}$-spectrum of the double star tree $S_{m,n}$ is
$$ \{ \left( \alpha(\frac{1}{2}n+\frac{1}{3}m+1)-\frac{1}{2}(1-\alpha) \right)^{[m-1]},
\left( \alpha(\frac{1}{2}m+\frac{1}{3}n+1)-\frac{1}{2}(1-\alpha) \right)^{[n-1]}, \theta_{1}, \theta_{2}, \theta_{3}, \theta_{4}   \},   $$
where $\theta_{i}$ $(i=1,2,3,4)$ are the eigenvalues of $U^{**}$
\end{proof}

\end{document}